\documentclass[10pt,reqno]{amsart}

\usepackage{amssymb}
\usepackage{amsfonts}
\usepackage{amsthm}
\usepackage{amsmath}
\usepackage{bbm}
\usepackage{graphicx} 
\usepackage{xcolor}
\usepackage{mathtools}

\usepackage[
maxbibnames=9,
]{biblatex}

\renewbibmacro{in:}{}

\numberwithin{equation}{section}
\allowdisplaybreaks

\newtheorem{theorem}{Theorem}[section]
\newtheorem{proposition}[theorem]{Proposition}

\newtheorem{lemma}[theorem]{Lemma}

\newtheorem{claim}[theorem]{Claim}
\theoremstyle{definition}
\newtheorem{definition}[theorem]{Definition}
\newtheorem{example}[theorem]{Example}

\theoremstyle{remark}

\renewcommand{\H}{\mathbb{H}}

\newcommand{\N}{\mathbb{N}}

\newcommand{\R}{\mathbb{R}}

\newcommand{\Z}{\mathbb{Z}}

\newcommand{\frakg}{{\mathfrak g}}

\renewcommand{\hat}{\widehat}
\newcommand{\eps}{\varepsilon}

\newcommand{\scriptA}{\mathcal{A}}

\newcommand{\scriptC}{\mathcal{C}}

\newcommand{\scriptI}{\mathcal{I}}
\newcommand{\scriptJ}{\mathcal{J}}

\newcommand{\scriptL}{\mathcal{L}}
\newcommand{\scriptM}{\mathcal{M}}

\newcommand{\scriptO}{\mathcal{O}}
\newcommand{\scriptP}{\mathcal{P}}

\newcommand{\scriptV}{\mathcal{V}}

\newcommand{\scriptX}{\mathcal{X}}

\title{Incidences among flows}

\author{Kaiyi Huang, Betsy Stovall, and Sarah Tammen}

\begin{document}

\begin{abstract}
We consider two incidence problems for integral curves of vector fields.  The first is an analogue of the Euclidean joints problem, in which lines are replaced by integral curves of smooth vector fields taken from some finite-dimensional set.  The second is a bilinear and more rigid version of the first, in which we have two vector fields and one family of integral curves tangent to each and wish to know how many intersecting pairs are possible.    In both cases, a curvature condition of H\"ormander makes possible nontrivial bounds on the number of incidences in terms of the number of curves.    
\end{abstract}

\maketitle

\section{Introduction} \label{S:Intro}

In this article, we consider variations on two well-known  results in geometric combinatorics, the Szemer\'edi--Trotter theorem, \cite{szemeredi1983extremal}, bounding the number of incidences between lines and points in the plane, and the joints theorem, \cite{guth2015erdHos, kaplan2010lines, quilodran2010joints}, bounding the number of $n$-fold transversal intersections of lines in $\R^n$.  Namely, we are interested in discrete analogues of $L^p$-boundedness results, both linear and multilinear, for certain \textit{generalized Radon transforms}, for which both curvature and transversality play a crucial role.   

To provide the analytic motivation for our results, let us consider a family of smooth submersions $\pi_1,\ldots,\pi_k:\R^n \to \R^{n-1}$.  These maps give rise to a multilinear form,
\begin{equation} \label{E:define M}
\scriptM(f_1,\ldots,f_k) := \int \prod_{j=1}^k f_j \circ \pi_j(x)\, dx,
\end{equation}
and we can ask for which exponents $p_1,\ldots,p_k$ does an inequality
\begin{equation} \label{E:bound M}
|\scriptM(f_1,\ldots,f_k)| \lesssim \prod_{j=1}^k \|f_j\|_{L^{p_j}}
\end{equation}
hold, uniformly over $f_j$ in the space of smooth functions with compact support in (say) the unit ball.  This question has largely been settled, \cite{TaoWright, stovall2011p, CDSS}, and the answer depends on whether and how certain vector fields tangent to the fibers $\pi_j^{-1}\{y_j\}$ and their iterated Lie brackets span the tangent space to $\R^n$ at each point.  
Boundedness of such a generalized Radon transform has a natural discrete analogue, most easily seen in the case of \eqref{E:bound M} when the $f_j$ are characteristic functions, $f_j = \chi_{E_j}$.   In this case, the left-hand side of \eqref{E:bound M} equals the measure of the incidence set
$$
\Omega:=\bigcap_{j=1}^k \pi_j^{-1}(E_j) = 
\{x : x \in \bigcap_{j=1}^k \pi_j^{-1}(y_j), \, \text{for some}\, y_j \in E_j, 1 \leq j \leq k\},
$$
and the right-hand side of \eqref{E:bound M}  equals the product of the measures of each of the sets of curves, $\{\pi_j^{-1}(y_j) : y_j \in E_j\}$, after projection.  We may naturally discretize this problem by replacing Lebesgue measure with counting measure on both sides.  Noting that the pre-image $\pi_j^{-1}(y_j)$ of a point under a submersion $\pi_j:\R^n \to \R^{n-1}$ may be locally represented as an integral curve of a vector field $X_j$ annhilated by $D\pi_j$, we arrive at an incidence problem for flows.  In the discretized case, it is natural to replace rigid families of curves by perturbations, which leads to an analogue of the joints problem, in which we ask how many transversal $k$-fold intersections can exist among a family of integral curves of vector fields lying in the span of $X_1,\ldots,X_k$.  As will be seen, both of these discrete variants of \eqref{E:bound M} include as sub-cases several widely studied questions in extremal combinatorics. 

In the continuum case, nontrivial estimates for the localized version of the problem (in which integration in \eqref{E:define M} is taken over a ball) are available \cite{TaoWright, stovall2011p} precisely when the vector fields obey a higher-order transversality hypothesis known as the H\"ormander condition, which we now state.
\begin{definition} \label{D:Hormander}
Let $\scriptA$ be a collection of smooth vector fields on $\R^n$.  We say that $\scriptA$ obeys the H\"ormander condition of order $N$ on $\R^n$ if the iterated Lie brackets of the form $[X_1,[X_2,[\cdots,[X_{m-1},X_m]\cdots]]]$, with $m \leq N$ and $X_j \in \scriptA$, span the tangent space to $\R^n$ at every point.   
\end{definition}
Our first result is a simple analogue of the joints theorem on $\R^n$, with transversality replaced by the H\"ormander condition.

\begin{theorem}[Joints arising from flows] \label{T:joints}
Let $\scriptV \subseteq \scriptX(\R^n)$ be a linear subspace of the set of smooth vector fields on $\R^n$.  Assume that the flow, $(t,x) \mapsto e^{tX}(x)$, of each element $X \in \scriptV$ is a polynomial of degree at most $N$ in $(t,x)$ and that $\scriptV$ obeys the H\"ormander condition of order $N$ on $\R^n$.  Let  $\scriptL$ be a finite collection of integral curves $\{e^{tX}(x_0) : t \in \R\}$ of vector fields $X \in\scriptV$.  We say that $p \in \R^n$ is a joint of $\scriptV$ if there exist $X_1,\ldots,X_k \in \scriptV$, whose span is $\scriptV$, such that $\{e^{tX_j}(p) : t \in \R\} \in \scriptL$, for $1 \leq j \leq k$.  If $\scriptJ$ is the set of joints formed by elements of $\scriptL$, then
\begin{equation} \label{E:joints}
\#\scriptJ \lesssim \# \scriptL^{\frac{n}{n-1}},
\end{equation}
where the implicit constant depends on $n,N$.  
\end{theorem}

We note, in particular, that the Euclidean joints theorem, proved in \cite{guth2015erdHos, kaplan2010lines, quilodran2010joints}, concerns the special case when $k=n$ and $\scriptV$ is the collection of all constant coefficient vector fields.  In that well-known case, the elements of $\scriptL$ are lines, and a joint is the intersection point of a collection of lines whose directions span the tangent space to $\R^n$.  In the general case, our transversality condition is weaker and our curves more general,  yet we arrive at the same result.


We turn now to a more rigid variant, a direct analogue of \eqref{E:bound M} in the case $k=2$.

\begin{theorem}[Bounds for restricted multijoints from flows] \label{T:multijoint}
Let $\pi_1,\pi_2:\R^n \to \R^{n-1}$ be polynomial maps, and let $X_1,X_2$ be two vector fields on $\R^n$ with $X_j\pi_j \equiv 0$.  Assume that the iterated commutators of $X_1,X_2$ generate a nilpotent Lie algebra $\mathfrak g$ of step at most $N$, and assume further that the elements of $\mathfrak g$ have polynomial flows and that their directions span the tangent space of $\R^n$ at every point.  Then if $\scriptL_j$ is a finite collection of integral curves of $X_j$, $j=1,2$, and $\scriptI(\scriptL_1,\scriptL_2)$ is the set intersection points of pairs of curves in $\scriptL_1 \times \scriptL_2$,  we have the estimate
\begin{equation}\label{E:multijoint}
\#\scriptI(\scriptL_1,\scriptL_2) \lesssim \#\scriptL_1 + \#\scriptL_2 + (\#\scriptL_1\#\scriptL_2)^{\frac{n-1}{2n-3}},
\end{equation}
where the implicit constant depends on $n,N$.
\end{theorem}

We can compare the results of Theorems~\ref{T:joints} and~\ref{T:multijoint} by noting that the homogeneous degree of the main term in \eqref{E:multijoint} is $1+\tfrac1{2n-3}$, which is significantly closer to 1 than the exponent $1+\frac1{n-1}$ from \eqref{E:joints}.

Finally, although we have stated our main results in the setting of Euclidean spaces, we also formulate and prove analogous results for vector fields on manifolds simply connected manifolds.  See Theorems \ref{T:joints in M} and \ref{T:multijoints in M} in Section~\ref{S:manifolds}.

\subsection*{Notation} In this article, we use the notation $A\lesssim B$ to mean that $A\leq CB$ for some finite constant $C>0$ that may depend on $N$.  We also use the notation $A\ll B$ to mean that $A\leq cB$ for some sufficiently small constant $c>0$ that may depend on $N$. 

\subsection*{Outline of paper}  In the next section, Section~\ref{S:Background}, we will give an overview of some motivating and related results from both the continuum and discrete settings and discuss some possible open questions suggested by these results.  In Section~\ref{S:Examples}, we will take a detailed look at some basic examples, and provide some counter-examples. In Section~\ref{S:Joints}, we will prove Theorem~\ref{T:joints} by generalizing the proof of the joints theorem from  \cite{quilodran2010joints} to utilize the H\"ormander condition rather than transversality.  Our proof of Theorem~\ref{T:multijoint} is inspired by the polynomial partitioning proof of the Szem\'eredi--Trotter theorem,\footnote{Various polynomial partitioning proofs of the Szem\'eredi--Trotter theorem were independently developed by multiple authors at essentially the same time.  We follow a version from Terence Tao's blog (credited therein to Netz Katz) \cite{taowebsite}; other versions may be found in \cite{guth2016polynomial, KMS}. While we do not intend to provide a comprehensive historical account, we would be remiss in not noting that these arguments have as ancestor the divide-and-conquer approach of \cite{clarkson1990combinatorial}.}
and, as such, our first step, carried out in Section~\ref{S:I in Variety}, is to prove an initial estimate on the number of incidences that can be contained in a degree $D$ variety.  We will then complete the proof of Theorem~\ref{T:multijoint} in Section~\ref{S:Multijoints}.   Finally, in Section~\ref{S:manifolds}, we state and prove generalizations of our main theorems in a more invariant setting.

\subsection*{Acknowledgements}
 We would like to thank Paul Pollack for a helpful conversation related to Example 3.1, Tony Carbery and Yannis Galanos for sharing the results in Galanos's thesis, and Larry Guth for a question about \cite{CDSS} that inspired this project.
 
This article is based on research carried out while the first and second authors were supported by NSF DMS-2246906; the first author was supported by a Bung-Fung Lee Torng Graduate Student Summer Fellowship and a UW-Madison Mathematics Hirschfelder Fellowship; the second author was supported by WARF; and the third author was supported by NSF DMS-2037851.  This work was also supported by a grant from the Simons Foundation [SFI-MPS-SFM-00011865, B.S.].

\section{Background and motivation} \label{S:Background}
\subsection{Continuum analogues}
Let $\pi_1,\dots,\pi_m:\R^n\longrightarrow\R^{n-1}$ be smooth (local) submersions, let $a:\R^n\longrightarrow[0,\infty)$ be a smooth cutoff function of compact support, and consider the multilinear form 
\begin{equation}\label{E:multilinear}
    \scriptM(f_1,\ldots,f_m):=\int_{\R^n}\displaystyle\prod_{j=1}^mf_j\circ\pi_j(x)a(x)dx.
\end{equation}
This is an example of a multilinear generalized Radon transform, and we consider its $L^p$ boundedness properties. As mentioned previously, bounding $\scriptM$ amounts to controlling the measure of the incidence set of the level curves of the $\pi_j$, and we can view these level curves as flows of vector fields $X_j$ annihilated by the $d\pi_j$, or, equivalently, under whose flow the $\pi_j$ are invariant.  (There is a natural way to construct these vector fields, namely, for each $j$, we take the wedge of the $n-1$ rows of $D\pi_j$ and then apply the Hodge-star operation obtain a vector field.)  

If there exists some positive codimensional submanifold $\Pi \subset\R^n$ such that the $X_j$ are all tangent to $\Pi$, then no ``interesting'' estimates are possible.  Indeed, since the projections $\pi_j(\Pi)$ all have the same codimension as $\Pi$, the condition $\tfrac1{p_1}+\cdots+\tfrac1{p_m}\leq 1$ is seen to be necessary.  However, this is precisely the condition under which \eqref{E:bound M} follows from H\"older's inequality.  Localizing, we obtain the same conclusion if the vector fields are merely tangent to $\Pi$ to infinite order at some point.  Thus we see the necessity of the H\"ormander condition, Definition~\ref{D:Hormander}, which we may think of as quantitative failure of the hypothesis of the Frobenius theorem.  

Conversely, by quantifying the ``order'' in H\"ormander's condition further, we may state an essentially sharp theorem in the continuum case.  For each finite sequence $w \in \{1,\ldots,m\}^k$, with $k \in \N$, we can form the iterated Lie bracket
$$
X_w:=[X_{w_1},[X_{w_2},[\ldots,[X_{w_{k-1}},X_{w_k}]\ldots]].
$$
We define $\deg w \in \Z_{\geq 0}^m$ to be the vector whose $j$-th entry equals the number of times that $j$ appears in $w$.  Under the H\"ormander condition, after shrinking the support of $a$ if needed (via a partition of unity), there exists some  $(w_1,\ldots,w_n)$ with $\det(X_{w_1},\ldots,X_{w_n}) \neq 0$.  We may then define exponents $(p_1,\ldots,p_m) \in [1,\infty]^m$ by
$$
p_j:=\frac{
\bigl(\sum_{i=1}^n\sum_{l=1}^m \deg_l w_i\bigr) - 1}{\sum_{i=1}^n \deg_j w_i}, \qquad 1 \leq j \leq m.
$$
Then the main results of \cite{stovall2011p, TaoWright} state that
\begin{equation} \label{E:multilinear exact}
|\scriptM(f_1,\ldots,f_m)| \lesssim \prod_{j=1}^m \|f_j\|_{L^{p_j+\eps}},
\end{equation}
for all $\eps>0$.  This result is sharp up to endpoints (after interpolation with other estimates of this form).  Under the hypotheses of Theorem~\ref{T:multijoint}, the $\eps$ loss can be removed when $k=2$ \cite{CDSS} or if the $f_j$ are characteristic functions \cite{gressman2009p, stovall2011p}. Special cases include the Loomis--Whitney inequality,
$$
\pi_j(x):=(x_1,\ldots,\widehat{x_j},\ldots,x_n), \qquad 1 \leq j \leq n,
$$
and convolution with arclength measure on the moment curve, for which we may take
$$
\pi_1(x',x_n) = x', \qquad \pi_2(x',x_n) = x'-(x_n,\ldots,x_n^{n-1}).  
$$

In their widest generality, these estimates have been proved using an iteration scheme, sometimes called the ``method of refinements,'' in which one forms a map from $\R^n$ to $\R^n$ by traveling first along a curve from one family, then, from each point of the curve, along a curve from another family, and so on.  The map is nearly one-to-one (thanks to the H\"ormander conditon), and growth of the Jacobian determinant (which typically vanishes at 0, but has certain derivatives nonzero) leads to the precise form of the estimate in \eqref{E:multilinear exact}.  In the discrete case, this map is essentially one-to-one (after some pruning), leading to nontrivial estimates, but with a worse exponent than that in Theorem~\ref{T:multijoint}.  The methods in this paper give an improved estimate over that approach, but, lacking an effective way of using the growth of the Jacobian, we are unable to match the exponents from the continuum case.

In particular, though we think of Theorem \ref{T:multijoint} as a discretized partial analogue of the results for the bilinear continuum case, the exponent $\frac{2n-2}{2n-3}$ arising in \eqref{E:multijoint} is larger than the sharp exponent in the continuum case unless $n=3$ and $X_1,X_2,X_{12}$ forms a spanning set of vector fields. Lacking a counter-example, it seems reasonable to ask whether the optimal exponents in \eqref{E:multijoint} arise from the endpoints of the Newton polytope of exponents of the corresponding Radon-like transforms, that is, whether the estimate
$$
\#\scriptI(\scriptL_1,\scriptL_2) \stackrel{?}{\lesssim} \#\scriptL_1 + \#\scriptL_2 + \#\scriptL_1^{1/p_1}\#\scriptL_2^{1/p_2} 
$$  
might hold, with $p_1,p_2$ corresponding to sharp exponents from the continuum case, though perhaps with a loss.  (We note that the continuum problem is made easier by raising the $p_j$, but the discrete problem is made easier by lowering the $p_j$.)  Another natural question is whether Theorem~\ref{T:multijoint} might be multilinearized, with $k$ families of curves rather than just 2.  (We will discuss a specific case for which this has been done in the next section.)  

Continuing the analogy with the continuum case, we can view Theorem~\ref{T:joints} as a discretized \textit{and perturbed} partial analogue of the fully multilinear continuum result. Thus, it is also natural to ask whether, when the vector field $\scriptV$ is generated by the vector fields corresponding to $\pi_1,\ldots,\pi_m$, \eqref{E:joints} might be improved to 
$$
\#\scriptJ \stackrel{?}{\lesssim} \#\scriptL^{\frac1{p_1}+\cdots+\frac1{p_m}}.
$$
However, in Subsection~\ref{SS:ST} we will show that this is not possible by providing a counter-example.  (That being said, it would still be natural to ask whether some improvement on \eqref{E:joints} might be possible when the dimension of $\scriptV$ is smaller than the ambient dimension.)

\subsection{A small sampling of prior discrete incidence bounds}  The combinatorial literature on point-curve and curve-curve incidences is vast, but seems not to directly address our questions, whose hypotheses are analytically, rather than combinatorially, motivated.  Here we provide a small sampling by way of comparison.  Before proceeding, we note that projection by $\pi_1$ transforms Theorem~\ref{T:multijoint} into a problem of counting incidences between a set of points $\scriptP := \pi_1(\scriptL_1)$ and a set of curves $\scriptC := \pi_1(\scriptL_2)$ in dimension $n-1$, and, likewise, Theorem~\ref{T:joints} can be projected (perhaps via some map naturally associated to the vector fields) to the problem of counting curve-curve incidences (under an appropriate definition of incidence).  We will look at some concrete examples in the next section.  

The point-curve incidence problem has been studied by Pach--Sharir \cite{PachSharir98} in dimension 2 and by Sharir--Solomon--Sheffer \cite{SharirSolomonSheffer16} in general dimensions under the hypothesis that the family of curves has finitely many degrees of freedom and finite multiplicity, conditions which reflect the number of points needed to determine a bounded number of curves, the number of intersection points of two curves, and, in higher dimensions, the number of curves that can lie in lower-dimensional varieties of bounded degree.  Related conditions also appear in 
\cite{sharir2017cutting,sharir2022incidences}.  However, the notions of degrees of freedom and multiplicity, while well-suited to a number of combinatorial problems, do not seem to reflect the full range of combinatorial constraints determined by the H\"ormander condition, and we can contrast the homogeneous degree of the main term in \eqref{E:multijoint}, which is $1+\frac1{2n-3}$, with that of the main theorem in \cite{SharirSolomonSheffer16}, which is $1+\frac{1}{n-1+1/(k-1)}$.

The question of what combinatorial hypothesis might both hold for vector fields obeying the nilpotency and H\"ormander hypotheses and also still lead to the estimates in our theorems, seems like an interesting one, which we have not yet explored.

In a more analytic vein, Iosevich--{\L}aba--Jorati \cite{iosevich2009geometric} have used Sobolev smoothing estimates for  averaging operators (which, by duality, can be expressed in the form \eqref{E:define M} with $m=2$) to obtain nontrivial point-curve (or, more generally, point-manifold) incidence results of a similar flavor to our Theorem~\ref{T:multijoint}, but only under fairly strong regularity conditions on the set of points and the set of curves.  

Likewise, a number of authors, motivated by questions in ergodic theory, have considered such problems in the setting of the integers.  Such results are typically of a significantly different flavor.  (See, for example, \cite{han2020improving}, and the references therein.)  Of course, working on the integers necessarily imposes some regularity and provides new, number theoretic, tools.  Moreover, these results concern $\ell^p$ inequalities, which must lack the linear terms appearing on the right of \eqref{E:multijoint}, and necessitating a larger total power (or else some other compensating factor).    

Our work is closest in spirit to, and directly motivated by, \cite{fassler2020planar}, which studies incidences of points within $\delta$-neighborhoods of $\delta$-separated lines in $\R^2$.  While we consider incidences of curves, rather than tubes, a higher dimensional analogue of the problem from \cite{fassler2020planar} is considered in forthcoming work of Yannis Galanos (personal communication).   

Finally, we reserve discussion of some very special cases, including discretizations of Loomis--Whitney, to the next section, where we can speak more concretely, in terms of examples.

\section{Examples, counter-examples, and questions} \label{S:Examples}

In this section, we will explore a few concrete examples in more detail to better connect this problem with prior results in the literature.  We begin with a case in which the continuum analogy works perfectly (and thus, perhaps, a bit too well).  

\subsection{Loomis--Whitney, Joints, and Multijoints} 

The Loomis--Whitney inequality states that if $0 \leq f_1, \dots, f_n \in L^1(\R^{n-1})$, then
\begin{equation} \label{E:LoomisWhitney}
\int \prod_{j=1}^n f_j^{\frac{1}{n-1}}(x_1,\ldots,\widehat{x_j},\ldots,x_n)\, dx \lesssim \prod_{j=1}^n \|f_j\|_1^{\frac1{n-1}}.
\end{equation}
Upon discretization, we have multisets of lines parallel to the coordinate directions (with multiplicities $f_j$), and we can ask whether 
\begin{equation} \label{E:elementary multijoints}
\sum_{p \in \scriptJ} m(p)^{\frac1{n-1}} \lesssim \prod_j \#\scriptL_j^{\frac1{n-1}},
\end{equation}
where $m(p)$ equals the multiplicity of $p$, i.e., the number of ways in which $p$ is a joint.  Indeed, inequality \eqref{E:LoomisWhitney} can be used directly to prove  \eqref{E:elementary multijoints} (and vice versa).  Furthermore, Ruixiang Zhang \cite{zhang2020proof} has proved that \eqref{E:elementary multijoints} holds for arbitrary multisets of lines in $\R^n$, not merely axis-parallel ones, and in three dimensions, we have even more, as Iliopoulou \cite{Iliopolou2015} has proved that in the 3D multijoints theorem, lines can be replaced with algebraic curves.

Thus the Loomis--Whitney inequality is quite stable under discretizations, both perturbed and not.  A natural question is whether this might be the case for more general inequalities of the form \eqref{E:bound M}.  
It is thus disappointing that neither the above-mentioned approach to the rigid discrete inequality \eqref{E:elementary multijoints}, nor the validity of \eqref{E:elementary multijoints} under perturbation seems to extend to the more general setting of \eqref{E:bound M}.  Indeed, the transference argument between  \eqref{E:elementary multijoints} in the rigid case and \eqref{E:LoomisWhitney} relies on the fact that incident lines have neighborhoods that are incident tubes, the shape and size of whose intersection does not depend on their location.  Worse, we will see that there can be no direct multijoints analogue of \eqref{E:bound M} in general.


\subsection{Szemer\'edi--Trotter and perturbations thereof}  \label{SS:ST}
Now we elaborate on the above-mentioned connection with Szemer\'edi--Trotter, which we first learned of through a manuscript of F\"assler--Orponen--Pinamonti \cite{fassler2020planar}.  Our setting is the first Heisenberg group, $\mathbb H^1 \simeq \R^3 \ni (x,y,t)$, with vector fields, 
\begin{align*}
    X &:= \tfrac{\partial}{\partial x} - \tfrac12 y \tfrac{\partial}{\partial t}\\
    Y &:= \tfrac{\partial}{\partial y} + \tfrac12 x \tfrac{\partial}{\partial t},
\end{align*}
having Lie bracket $[X,Y] = \tfrac{\partial}{\partial t}=:T$ and flows
$$
e^{sX}(x,y,t) = (x+s,y,t-\tfrac12 sy), \qquad e^{sY}(x,y,t) = (x,y+s,t+\tfrac12 sx),
$$
and annihilated by the projections
$$
\pi_X(x,y,t):=(y,t+\tfrac12xy), \qquad \pi_Y(x,y,t):=(x,t-\tfrac12xy),
$$
respectively.  Of course, $\pi_X$ projects integral curves of $X$ to points; it also projects integral curves of $Y$ to lines:
$$
\pi_X(x,y+s,t+\tfrac12 sx) = (y+s,t+sx +\tfrac12xy).  
$$
Thus the problem of counting incidences between integral curves of $X$ and integral curves of $Y$ in $\R^3$ maps exactly to the question of counting incidences between points and lines in the plane.  Moreover, we see that in three dimensions, Theorem~\ref{T:multijoint} is sharp (up to the precise behavior of the implicit constant).  (Indeed, in this special case, as noted previously, our proof of Theorem~\ref{T:multijoint} reduces to a known proof \cite{taowebsite} of the Szemer\'edi--Trotter theorem.)  

We now consider a perturbation, in which we wish to count incidences between integral curves of $X$ and integral curves of linear combinations of $X$ and $Y$.  For $\omega \in \mathbb S^1 \subseteq \R^2$, we set $X_\omega:= \omega_1 X + \omega_2 Y$, and see that 
\begin{gather*}
e^{sX_\omega}(x,y,t) = (x+\omega_1 s, y+\omega_2 s, t + \tfrac12(\omega_2 x s - \omega_1 sy)), \\ \pi_X(e^{sX_\omega}(x,y,t)) = (y+\omega_2 s, t+\tfrac12 xy + \omega_2 xs + \tfrac12 \omega_1\omega_2 s^2).
\end{gather*}
For $\omega = (\cos\theta,\sin\theta)$ and $\sin\theta \neq 0$, this parabola can be reparametrized as
\begin{equation} \label{E:parab param}
u \mapsto (y+u,t+\tfrac12 xy + xu + \tfrac12 \cot \theta u^2).  
\end{equation}
Thus, counting incidences between integral curves of $X$ and exponentiations of the $X_\omega$ (which are lines) projects to the problem of counting incidences between points and arbitrary parabolas in the plane.  

However, the direct analogue of the Szem\'eredi--Trotter estimate fails to bound the number of point-parabola incidences in the plane.  Indeed, let $\scriptP$ denote the collection of points in the array $[0,N] \times [0,3N^3]$, and let $\scriptC$ denote the collection of 
parabolas of the form
$$
p_{(a,b,c)}:u \mapsto (u,a+b u + c u^2),
$$
with $(a,b,c) \in \Z^3 \cap [0,N^3] \times [0,N^2] \times [0,N]$.  There are $N^6$ parabolas, each incident to $N$ points, yielding $N^7$ total incidences.  On the other hand, 
$$
(\#\scriptP \#\scriptC)^{\frac23} \sim N^{\frac{20}3},
$$
which is much smaller than $N^7$ for $N$ large.  From this example, we see that the bound in Theorem~\ref{T:multijoint} can fail under perturbations of the lines, in contrast with the case of Euclidean multijoints.  (To our knowledge, the best-known bound for the number of incidences of points and parabolas is 
\begin{equation} \label{E:PS-SZ}
\#\scriptI \lesssim \#\scriptP + \#\scriptC + \min\{\#\scriptP^{\frac 35}\#\scriptC^{\frac45},\#\scriptP^\frac{6}{11}\#\scriptC^{\frac{9}{11}+\eps}+(\#\scriptP\#\scriptC)^{\frac{2}{3}}\}
\end{equation}
due to Pach--Sharir \cite{PachSharir98} and Sharir--Zahl \cite{sharir2017cutting}.)    

Finally, we consider the fully general case:  $\scriptL$ is a collection of integral curves of vector fields of the form $X_\omega$.  Projecting via $\pi_1$ maps the problem of counting joints formed by elements of $\scriptL$ to counting tangential intersections of parabolas.  Theorem~\ref{T:joints} gives the estimate 
\begin{equation} \label{E:tangential parab}
\#\scriptJ \lesssim \#\scriptC^{\frac 32},
\end{equation}
which matches the estimate obtained by Ellenberg--Solymosi--Zahl in \cite{ellenberg2016new}.  

While we have no reason to think that \eqref{E:tangential parab} should be sharp, we now turn to an example showing that the exponent $\frac32$ cannot be replaced with $\frac43$, which we recall is the homogeneous degree of the main term in the unperturbed problem. (Comparing exponents with those in \eqref{E:PS-SZ}, we observe that $\frac32 > \frac35+\frac45  > \frac{6}{11}+\frac9{11}> \frac43=\frac23+\frac23$.)  Before giving our example, we note that the standard sharp example for Szem\'eredi--Trotter shows that $\#\scriptI$ can be as large as $\#\scriptL^{4/3}$; we will improve on that by a logarithmic factor.  

\begin{example}
\label{tangentialIntersections}
We consider again the set 
$$
\scriptC:=\{p_{(a,b,c)} : (a,b,c) \in \Z^3 \cap [0,N^3] \times [0,N^2] \times [0,N]\},
$$ 
defined above.  

We observe that two parabolas
$p_{(a,b,c)}$ and $p_{(a',b',c')}$ are incident (intersect tangentially) if and only if $p_{(a-a',b-b',c-c')}$ and $p_{(0,0,0)}$ (the $x$-axis) are incident, which holds if and only if $4(a-a')(c-c') = (b-b')^2$.  Thus, for each triple $(a,b,c)$ for which $4ac=b^2$, there are about $N^6$ incidences in $\scriptC$.  (Since there are $N^6$ ways to express a triple $(a,b,c)$ as a difference of parameters of two curves in $\scriptC$.)    

Now we turn to the problem of counting the number of solutions $(a,b,c)$ to $4ac=b^2$.   We observe that $4ac$ is square if and only if $4a/\psi(c)$ is a square integer, where $\psi(c)$ is the square-free part of $c$.  Thus the number of pairs $(a,c) \in [0,N^3] \times [0,N]$ such that $4ac$ is square is about
\begin{align*}
\sum_{c=1}^N &\#\{\text{squares in $[0,N^3/\psi(c)]$}\} 
\gtrsim
\sideset{}{'}\sum_{\psi=1}^N\#\{c : \psi(c)=\psi\}\bigl(\tfrac{N^3}{\psi}\bigr)^{1/2}\\
&\sim
\sideset{}{'}\sum_{\psi=1}^N\#\{\text{squares in $[0,N/\psi]$}\}\bigl(\tfrac{N^3}{\psi}\bigr)^{1/2}
=
\sideset{}{'}\sum_{\psi=1}^N N^2 \tfrac1{\psi},
\end{align*}
where the ${}'$ indicates that the sum is taken over square-free $\psi$.  We may lower-bound the sum of reciprocals of square-free integers via an elementary argument.\footnote{An earlier version of this manuscript gave  a double-logarithmic loss by using Mertens second theorem \cite{Mertens} and the fact that all primes are square-free.  We are grateful to Paul Pollack for pointing out to us the improved version presented here.}   Indeed, we have
\begin{align*}
\log N \sim \sum_{n=1}^N \tfrac1n \leq \bigl(\sum_{n=1}^\infty \tfrac1{n^2}\bigr)\bigl(\sideset{}{'}\sum_{\psi=1}^N \tfrac1\psi\bigr) \lesssim \sideset{}{'}\sum_{\psi=1}^N \tfrac1\psi.
\end{align*}
In all, the $N^6$ parabolas in $\scriptC$ can have at least $N^8 \log N$ incidences, much larger than $\#\scriptC^{4/3}$.
\end{example}

\subsection{Point-curve and curve-curve incidences in higher dimensions}  The following example is motivated by the operator given by convolution with  arclength measure on a curve $\gamma:\R \to \R^d$,
$$
Af(x) = \int_\R f(x+\gamma(t))\, dt.
$$
For simplicity, we consider the case $\gamma(t) = (t,\ldots,t^d)$.  By work of Christ \cite{christ1998convolution},if $E,F \subseteq\R^d$ are measurable sets, then the set of incidences between points of $E$ and translates of $\gamma$ by $F$, viewed as a subset of $\R^n:=\R^{d+1}$, has measure given by 
$$
\langle \chi_E,A^*\chi_F\rangle \lesssim |E|^{\frac1p}|F|^{\frac1q},
$$
for $(\tfrac1p,\tfrac1q)$ lying on the closed line segment joining the points
$$
\bigl(\frac{n^2-3n+4}{n(n-1)},\frac2n\bigr), \qquad \bigl(\frac2n,\frac{n^2-3n+4}{n(n-1)}\bigr).  
$$

We may reframe this problem as an incidence problem for flows via the lifting in \cite{TaoWright}.  Namely, 
 we define vector fields on $\R^n = \R^d \times \R \ni (x,t)$ by 
$$
X_1 = \tfrac{\partial}{\partial t}, \qquad X_2 = \tfrac{\partial}{\partial t} - \gamma'(t) \cdot \nabla_x.
$$
These are annhilated by
$$
\pi_1(x,t) = x, \qquad \pi_2(x,t) = x+\gamma(t),
$$
and have flows
$$
e^{sX_1}(x,t) = (x,t+s), \qquad e^{sX_2}(x,t) = (x+\gamma(t) - \gamma(t+s),t+s).
$$
After projection by $\pi_1$, Theorem~\ref{T:multijoint} bounds the number of incidences between a family $\scriptP_1$ of points in $\R^d$ and a family $\scriptC_1$ of translates of $\gamma$ by 
\begin{equation} \label{E:points moment curves}
\#\scriptI \lesssim \#\scriptP_1 + \#\scriptC_1 + (\#\scriptP_1 \#\scriptC_1)^{\frac{n-1}{2n-3}}.
\end{equation}
By comparison, results of Sharir--Solomon--Sheffer \cite{SharirSolomonSheffer16} yield the upper bound 
$$
\#\scriptI \lesssim \#\scriptP_1 + \#\scriptC_1 + \#\scriptP_1^{\frac2n}\#\scriptC_1^{\frac{n-1}{n}} + \rm{higher \,order \,terms},
$$
and we compare homogeneous degrees of the main terms, noting $\frac{2(n-1)}{2n-3} \leq \frac{n+1}n$, for $n \geq 3$, with strict inequality for $n \geq 4$. 
(To apply the results of \cite{SharirSolomonSheffer16}, one needs to verify that  the collection of translates of the moment curve has two degrees of freedom with multiplicity 1.)    In the special case $n=4$, the estimate in \eqref{E:points moment curves} can be shown to match that of Sharir--Solomon--Zlydenko \cite{sharir2022incidences}, but we believe our result is new in higher dimensions.  

We now translate the perturbed version of this problem,  Theorem~\ref{T:joints}, to this context.  We consider a finite collection of integral curves of linear combinations of $X_1,X_2$, and a joint is a point of  intersection of any two distinct curves.  Projecting via $\pi_1$, we obtain a family $\scriptC$ of curves in $\R^d$ parametrized as 
$$
\gamma_{a,y_0,t_0}(u)
=
\begin{cases}
    y_0 - a\gamma(t_0+u), \qquad &a \neq \infty\\
    y_0-\gamma'(t_0)u, \qquad &a = \infty.
\end{cases}
$$
A joint likewise projects to a  tangential intersection of a pair of such curves.   (A one-point curve is considered tangent to each curve intersecting it.)  Thus Theorem~\ref{T:joints} gives the estimate
$$
\#J \lesssim \#\scriptC^{\frac{d+1}d}.
$$
Though we expect that this bound is not sharp, we believe it to be new and note that it improves as $d$ increases.

\subsection{Restricted X-ray transforms}

Our next example comes from the restricted X-ray transform,
$$
Tf(x,t) = \int_\R f(s,x+s \gamma(t))\, ds,
$$
averaging $f:\R^{n-1} \to \R$ along the line with direction vector $(1,\gamma(t))$; here $\gamma:\R \to \R^{n-2}$ is a curve with nonvanishing torsion, such as the moment curve $\gamma(t) = (t,\ldots,t^{n-2})$.  

In $\R^n = \R^{n-2} \times \R \times \R \ni (x,s,t)$, we define vector fields
$$
X_1 := \tfrac{\partial}{\partial s}, \qquad X_2:= \tfrac{\partial}{\partial t} - s\gamma'(t) \cdot \nabla_x,
$$
which are annhilated by projections
$$
\pi_1(x,s,t) = (x,t), \qquad \pi_2(x,s,t) = (s,x+s\gamma(t)),
$$
and have integral curves
$$
e^{uX_1}(x,s,t) = (x,s+u,t), \qquad 
e^{uX_2}(x,s,t) = (x+s \gamma(t) - s\gamma(t+u),s,t+u),
$$
Noting that $\pi_2(e^{uX_1}(x_0,s_0,t_0)) = (s_0+u,x_0+(s_0+u)\gamma(t_0))$, projection maps the incidence problem in Theorem~\ref{T:multijoint} onto the incidence problem for points in $\R^{n-1}$ and the one-parameter family of lines in $\R^{n-1}$ whose direction vectors are contained in $\{(1,\gamma(t)):t \in \R\}$.  The bound in Theorem~\ref{T:multijoint} is 
$$
\#\scriptI \lesssim \#\scriptL_1 + \#\scriptL_2 + (\#\scriptL_1\#\scriptL_2)^{\frac{n-1}{2n-3}}.
$$

By contrast, the joints problem (for integral curves of any vector fields of the form) 
$$
\cos\theta_0 X_1  + \sin\theta_0 X_2
$$
is mapped via $\pi_2$ to the problem of counting tangential intersections of curves of the form
$$
w \mapsto (w,z_0 + \cot \theta_0 \Gamma(\tan\theta_0 w + a_0)),
$$
with $\theta_0 \in \R$ and $\Gamma$ an antiderivative of $\gamma$.

\section{Proof of Theorem~\ref{T:joints}, the generalized joints theorem} \label{S:Joints}

This section will primarily be devoted to the proof of Theorem~\ref{T:joints}, for which we will adapt the proof used by Quilodr\'an in \cite{quilodran2010joints}.  

Before proceeding to the proof, we remark that we cannot replace the criterion that at a joint, the vector fields generating the intersecting curves must span $\scriptV$, with the weaker condition that at a joint, the relevant vector fields obey the H\"ormander condition.  Informally, this is to be expected because one curve could arise as the flow of many different vector fields, but for concreteness, we provide an  example in the Heisenberg group $\H^1$.  Let us take $\scriptV$ to be the linear span of $\{X,Y,T=[X,Y]\}$, where $X=\partial_x-\frac{1}{2}y\partial_t,Y=\partial_y+\frac{1}{2}x\partial_t,T=\partial_t$. Then  every line can be realized as an integral curve of a left-invariant vector field in $\scriptV$. Furthermore, a pair of vector fields in $\scriptV$ obeys the H\"ormander condition if their orthogonal projections to the span of $\{X, Y\}$ are transverse. Letting $\scriptL$ consist of $2N$ lines in the $xy$-plane, with half parallel to the $x$-axis and half parallel to the $y$-axis, each intersecting pair obeys the H\"ormander condition, and there are $N^2 \sim \#\scriptL^2 $ intersecting pairs, which is comparable with the elementary estimate and much larger than $\#\scriptL^{\frac32}$.  

Now we turn to the proof of Theorem~\ref{T:joints}, for which the key step is the following lemma.  

\begin{lemma} \label{L:Quilodran}
Let $m$ be an integer and $\scriptL$ a set of integral curves of vector fields in $\scriptV$, and suppose that some nonempty subset $\emptyset \neq \scriptJ' \subseteq \scriptJ$ of the joints of $\scriptL$ has the property that $\#\ell \cap \scriptJ' \geq m$ whenever $\ell \in \scriptL$ and $\ell \cap \scriptJ' \neq \emptyset$.  Then $\#\scriptJ' \geq c_N m^n$.  
\end{lemma}

\begin{proof}
    If $\#\scriptJ' < c_N m^n$, there exists a polynomial $Q\not\equiv 0$ of degree $\deg Q < c_N m$ such that $Q \equiv 0$ on $\scriptJ'$.  Without loss of generality, $Q$ has minimal degree among all such polynomials.  
    
    Let $\scriptL'$ denote the set of curves in $\scriptL$ intersecting $\scriptJ'$, and let $\ell=:\{e^{tX_\ell}(x) : t \in \R\} \in \scriptL'$.  Then $t\mapsto Q(e^{tX_\ell}(x))$ is a polynomial of degree at most $N \deg Q$, vanishing at least $m$ times so, provided $c_N \leq N^{-1}$, $Q \equiv 0$ on $\ell$.  Therefore $X_lQ \equiv 0$ on $\ell$, for every $\ell \in \scriptL'$.  

    At a joint $p \in \scriptJ'$, $XQ(p) = 0$ for every $X \in \scriptV$. Each $XQ$ is a polynomial of degree at most $N \deg Q$, vanishing on $\scriptJ'$, so, by the same argument, $X_1X_2 Q(p) = 0$, for every $p \in \scriptJ'$ and $X_1,X_2 \in \scriptV$, provided $c_N < N^{-2}$.  Continuing in this fashion, as long as $c_N < N^{-N}$, $X_1 \cdots X_N Q(p) = 0$, for every $X_i \in \scriptV$ and $p \in \scriptJ'$.  

    Therefore $YQ \equiv 0$ on $\scriptJ'$, for every element of the Lie algebra generated by elements of $\scriptV$.  Since $\scriptV$ obeys the H\"ormander condition of order $N$, $\nabla Q \equiv 0$ on $\scriptJ'$.  Since $\nabla Q$ has lower degree than $Q$ and vanishes on $\scriptJ'$, by hypothesis, $\nabla Q \equiv 0$.  Thus $Q$ is a constant.  As $\scriptJ' \neq \emptyset$ and $Q$ vanishes on $\scriptJ'$, $Q \equiv 0$, a contradiction.  
\end{proof}

Now we may conclude the proof of Theorem~\ref{T:joints}.  

\begin{proof}
Let $m:=\tfrac12\#\scriptJ/\#\scriptL$, half the average number of joints per curve.  We wish to show that $\#\scriptJ \gtrsim m^n$, and so we proceed by contradiction, assuming initially that $\#\scriptJ < c_N m^n$, where $c_N$ is a sufficiently small constant.  

By Lemma~\ref{L:bad line},  there exists a curve $\ell \in \scriptL$ such that $\ell \cap \scriptJ < m$.  Let $
\scriptL_1 := \scriptL \setminus \{\ell\}$, let $\scriptJ_1$ denote the set of joints formed from $\scriptL_1$, and repeat.  Namely, given $i$ and $\emptyset \neq \scriptJ_i \subseteq \scriptJ$, which are the joints from $\scriptL_i \subseteq \scriptL$, we have  $\#\scriptJ_i \leq \#\scriptJ < c_N m^n$, so there exists $\ell \in \scriptL_i$ such that $\#\ell \cap \scriptJ_i < m$.  We remove $\ell$ from $\scriptL_i$ to form $\scriptL_{i+1}$, and denote by $\scriptJ_{i+1}$ the corresponding set of joints.  The process stops when $\scriptJ_i = \emptyset$.  By construction, $\#\scriptJ_i \geq \#\scriptJ-mi$, so at least $\#\scriptJ/m = 2\#\scriptL$ iterations are needed before we run out of joints.  However, we remove one curve at each step, so we run out of curves before joints, which is impossible (since curves are needed to form joints).  
\end{proof}

\section{Incidences in a variety} \label{S:I in Variety}

In this section, we prove a key proposition that will enable us to adapt the polynomial partitioning proof of Szem\'eredi--Trotter to prove Theorem \ref{T:multijoint}.  Under the hypotheses of  Theorem \ref{T:multijoint}, 
$X_1 \pi_1 \equiv 0$, 
so
\begin{equation} \label{E:P1}
\scriptP_1:=\{\pi_1(\ell_1) : \ell_1 \in \scriptL_1\}
\end{equation}
is a set of points in $\R^{n-1}$.  

\begin{proposition} \label{P:bilinear poly}
Suppose there exists $0 \not\equiv q$,  a polynomial of degree at most $D$ on $\R^{n-1}$, vanishing identically on $\scriptP_1$.  Then 
\begin{equation} \label{E:bilinear poly}
\#\scriptI(\scriptL_1,\scriptL_2) \lesssim \#\scriptL_1 + D\#\scriptL_2.
\end{equation}
\end{proposition}

\begin{proof}
We may, of course, assume that $\scriptL_1$ and $\scriptL_2$ are both nonempty; otherwise the incidence set is empty, in which case the estimate is trivial.  We may further assume that $q$ has minimal degree among all nonzero polynomials vanishing identically on $\scriptP_1$.  

We now suppose by way of contradiction that 
\begin{equation} \label{E:bilin poly false}
\#\scriptI > C \#\scriptL_1 + C D \#\scriptL_2,
\end{equation}
with $\scriptI:=\scriptI(\scriptL_1,\scriptL_2)$, $C$ a large constant, depending on $N,n$, to be determined.  

We say that a curve $\ell_1 \in \scriptL_1$ is \textit{bad} if $\ell_1$ has \textit{at most} $C$ incidences with $\scriptL_2$, and that a curve $\ell_2 \in \scriptL_2$ is \textit{bad} if it has \textit{at most} $CD$ incidences with $\scriptL_1$.  As we prove below, \eqref{E:bilin poly false} implies that there exists at least one bad curve.

\begin{claim}\label{L:bad line}
    Under hypothesis \eqref{E:bilin poly false}, there exists a bad curve.  
\end{claim}

\begin{proof}[Proof of Claim~\ref{L:bad line}]
We suppose, by way of contradiction, that every $\ell_1 \in \scriptL_1$ has more than $C$ incidences with curves in $\scriptL_2$ and every $\ell_2 \in \scriptL_2$ has more than $CD$ incidences with curves in $\scriptL_1$.  Since we may take $C > 1$ (it will be much larger, in fact) this assumption immediately implies that $d\pi_1$ has full rank on all curves in $\scriptL_1$.  Indeed, by construction, $X_1=0$ whenever $d\pi_1$ has rank smaller than $n-1$, an $\ell_1$ on which $X_1$ vanishes is a point, and a point can only lie on a single integral curve of $X_2$, hence at most one curve in $\scriptL_2$.  

We set $Q:=q \circ \pi_1$.  Therefore, $\deg Q \lesssim D$.  Since $Q$ vanishes identically on each curve of $\scriptL_1$, $Q \circ \ell_2$ has more than $CD$ zeros for every $\ell_2 \in \scriptL_2$.  Since $\deg(Q \circ \ell_2) \lesssim D$, for $C$ sufficiently large, $Q \equiv 0$ on every curve in $\scriptL_2$, and, since $(Q \circ \ell_2)' = (X_2Q) \circ \ell_2 \equiv 0$, we see that $(X_2^k Q) \circ \ell_2 \equiv 0$ for every $\ell_2 \in \scriptL_2$.  

Next, we would like to repeat the preceding argument on the curves in $\scriptL_1$, with $X_2^kQ$ in place of $Q$. 
However, since curves in $\scriptL_1$ might only have $C \ll D$ incidences with curves in $\scriptL_2$, we need stronger control on the degree of $X_2^k Q \circ \ell_1$.  
Specifically, we will use the following lemma, whose proof we postpone to  \S \ref{lemmaSubsection}.

\begin{lemma}\label{L:X1X2Qell1}
Let $m \geq 0$ and $j_1,k_1,\ldots,j_m,k_m \geq 0$ be integers, and let $\ell_1 \in \scriptL_1$.  Then the composition
\begin{equation}\label{E:X1X2Qell1}
    [X_1^{j_1}X_2^{k_1} \cdots X_1^{j_m}X_2^{k_m}Q] \circ \ell_1
\end{equation}
is a polynomial of degree $\lesssim K:=k_1+\cdots+k_m$.  
\end{lemma}

Now we complete the proof of Claim~\ref{L:bad line}, using Lemma~\ref{L:X1X2Qell1} (proved in \S\ref{lemmaSubsection}).  We have seen that for all $k \geq 0$, $X_2^kQ$ vanishes identically on every curve in $\scriptL_2$.  Thus, $(X_2^kQ) \circ \ell_1$ has more than $C$ zeros for every $\ell_1 \in \scriptL_1$.  By Lemma~\ref{L:X1X2Qell1}, for $k \lesssim N$ and $C$ sufficiently large, $(X_2^kQ)\circ \ell_1$, along with all of its derivatives, is identically zero.  Therefore, $X_1^j X_2^kQ$ has at least $CD$ zeros, and hence vanishes identically, on each curve in $\scriptL_2$ for all $j \geq 0$ and all $0 \leq k \lesssim N$.  Iterating the preceding argument, 
$$
X_1^{j_1}X_2^{k_1} \cdots X_1^{j_m}X_2^{k_m}Q
$$
vanishes identically on each curve of $\scriptL_1\cup \scriptL_2$ for all $k_1 + \cdots + k_m \lesssim N$.  

Expanding Lie brackets as linear combinations of iterated applications of vector fields, $XQ \equiv 0$ on $\scriptL_1$ for every $X \in \mathfrak g$.  By our hypothesized H\"ormander condition, $\nabla Q \equiv 0$ on each curve of $\scriptL_1$.  Recalling that $Q=q \circ \pi_1$ and that $d\pi_1$ has full rank on $\scriptL_1$, $\nabla q \equiv 0$ on $\pi_1(\ell_1)$ for every $\ell_1 \in \scriptL_1$.  However, we have hypothesized that $q$ has minimal degree among all nontrivial polynomials vanishing on $\scriptP_1$, so $\nabla q \equiv 0$, whence $q$ is a nonzero constant, a contradiction since $\scriptL_1 \neq \emptyset$.  This completes the proof of Claim~\ref{L:bad line}.  
\end{proof}

Having proved Claim \ref{L:bad line}, we now introduce a stopping-time procedure that involves iteratively removing bad curves until there are none left.  Initiating the first step with the collections $\scriptL_1,\scriptL_2$, each step of our procedure begins with collections $\scriptL_1' \subseteq \scriptL_1$ and $\scriptL_2' \subseteq \scriptL_2$, obeying
$$
\#\scriptI(\scriptL_1',\scriptL_2') > C\#\scriptL_1' + CD\#\scriptL_2'.  
$$
By Claim~\ref{L:bad line}, there exists a bad curve $\ell \in \scriptL_1' \cup \scriptL_2'$, and we set $\scriptL_j'':=\scriptL_j' \setminus \{\ell\}$.  (By the H\"ormander condition, $\scriptL_1 \cap \scriptL_2 = \emptyset$, for otherwise, $X_1,X_2$ and all of their Lie brackets are tangent to a proper submanifold of $\R^n$.)   We stop the procedure if 
\begin{equation} \label{E:stop}
\#\scriptI(\scriptL_1'',\scriptL_2'') \leq C\#\scriptL_1'' + CD\#\scriptL_2'';
\end{equation}
otherwise, we repeat the procedure with these new collections of curves.  In particular, if we run out of curves in either collection, the procedure must stop, because then the left-hand side of \eqref{E:stop} equals zero.  

For $j=1,2$, let $\scriptL_j''$ be the set of  curves remaining in $\scriptL_j$, and let $d_j$ denote the number of curves removed from $\scriptL_j$.  That is, $d_j:=\#\scriptL_j-\#\scriptL_j''$, at the stopping time.  Thus, $d_1 \leq \#\scriptL_1$ and $d_2 \leq \#\scriptL_2$, and, by parity considerations, at least one of these inequalities must be strict.  

By the definition of bad curves, removal of a bad $\ell_1$ results in a loss of at most $C$ incidences and removal of a bad $\ell_2$ results in a loss of at most $CD$ incidences, so
$$
\#\scriptI(\scriptL_1'',\scriptL_2'') \geq \#\scriptI(\scriptL_1,\scriptL_2) - Cd_1-CDd_2.  
$$
On the other hand, we have assumed \eqref{E:stop} and \eqref{E:bilin poly false}, so
\begin{align*}
C(\#\scriptL_1 - d_1) &+ CD(\#\scriptL_2-d_2)  = C \#\scriptL_1'' + CD\#\scriptL_2''
\geq \#\scriptI(\scriptL_1'',\scriptL_2'')\\
&> C(\#\scriptL_1-d_1) + CD(\#\scriptL_2-d_2) ,
\end{align*}
which is impossible.  This completes the proof of Proposition~\ref{P:bilinear poly}.
\end{proof}

\subsection{A degree bound for derivatives of $Q$ on curves $\ell_1$} 
\label{lemmaSubsection}

Now we give the postponed proof of Lemma~\ref{L:X1X2Qell1}, stated in the proof of Proposition~\ref{P:bilinear poly}.

\begin{proof}[Proof of Lemma~\ref{L:X1X2Qell1}]
It suffices to show that for all $j_1$ sufficiently large, we have
\begin{equation}
\label{identicallyZero}
X_1^{j_1} X_2^{k_1} \dots X_1^{j_m} X_2^{k_m} Q \equiv 0.
\end{equation}
Specifically, we claim that if $j_1 > NK$, then (\ref{identicallyZero}) holds.
Proving this claim is sufficient to conclude that 
the composition 
\[
(X_1^{j_1} X_2^{k_1} \dots X_1^{j_m} X_2^{k_m} Q) \circ \ell_1
\]
is a polynomial of degree $\lesssim K$.

Since $X_1X_2 = [X_1, X_2] + X_2 X_1$ and $X_1 Q \equiv 0$, we can write
$X_1^{j_1} X_2^{k_1} \dots X_1^{j_m} X_2^{k_m}Q$ as a sum
\begin{equation}
\label{sumRep}
\sum_{w = (w_1, \dots, w_K)} c_w X_{w_1} \dots X_{w_K} Q,
\end{equation}
where for each $i = 1, \dots K$, the vector field
$X_{w_i}$ is an iterated Lie bracket of the form
\[
X_{w_i} = [X_1, [X_1, \dots [X_1, X_2]]],
\]
containing $d_{w_i}$-many copies of $X_1$, and  
the sum is over $K$-tuples $(w_1, \dots, w_K)$ so that $d_{w_1} + \dots + d_{w_K} = j_1 + \dots + j_m$.
By the pigeonhole principle, for each $K$-tuple $(w_1,\ldots,w_K)$, there exists $i$ such that  
\begin{equation}
\label{degreeBig}
d_{w_i} \geq \frac{j_1 + \dots + j_m}{K} \geq \frac{j_1}{K}.
\end{equation}
If $j_1 > NK$, then, for such an $i$, we have $d_{w_i} > N$, which means that
\[
X_{w_{i}} \equiv 0,
\]
by the nilpotency condition.
Since we can find an $i$ satisfying (\ref{degreeBig}) for every term in the sum (\ref{sumRep}), it follows that if $j_1 > NK$, then
\[
X_1^{j_1} X_2^{k_1} \dots X_1^{j_m} X_2^{k_m} Q \equiv 0,
\]
as claimed.
\end{proof}

\section{Bounds for restricted multijoints} \label{S:Multijoints}

In this section, we prove Theorem~\ref{T:multijoint}. To do so, we will adapt the polynomial partitioning proof of the Szemer\'edi--Trotter theorem from \cite{taowebsite}.  

We first note an immediate consequence of Proposition~\ref{P:bilinear poly}.  

\begin{proposition}\label{P:simple est}
Under the hypotheses of Theorem~\ref{T:multijoint},
\begin{equation} \label{E:simple est}
\#\scriptI(\scriptL_1,\scriptL_2) \lesssim \min\{\#\scriptL_1 +\#\scriptL_1^{\frac1{n-1}}\#\scriptL_2, \#\scriptL_2 +\#\scriptL_1\#\scriptL_2^{\frac1{n-1}}\}.  
\end{equation}    
\end{proposition}

\begin{proof}
%
With $\scriptP_1$ as in \eqref{E:P1}, we recall that $\scriptP_1$ is a collection of points in $\R^{n-1}$, and $\#\scriptP_1 \leq \#\scriptL_1$.  By standard parameter counting arguments, there exists a polynomial $q \not \equiv 0$ of degree $\deg q \lesssim (\#\scriptP_1)^{\frac1{n-1}}$ that vanishes on $\scriptP_1$.  Thus by Proposition~\ref{P:bilinear poly},
$$
\#\scriptI(\scriptL_1,\scriptL_2) \lesssim \#\scriptL_1 + \#\scriptL_1^{\frac1{n-1}}\#\scriptL_2.  
$$
Since our hypotheses are symmetric, we obtain \eqref{E:simple est}.
\end{proof}

We are finally ready to complete the proof of Theorem~\ref{T:multijoint}.  

\begin{proof}[Proof of Theorem~\ref{T:multijoint}]
We first consider a degenerate case.  Let $\scriptL_j^0$ denote the set of all one-point ``curves'' of $\scriptL_j$; these are points at which $X_j$ vanishes.  We note that there are no incidences between $\scriptL^0_1$ and $\scriptL^0_2$, because at a point where $X_1$ and $X_2$ both vanish, all Lie brackets of $X_1,X_2$ are easily seen (by induction) to be zero.  

Since an element of $\scriptL_1^0$, being a singleton, can only intersect a single integral curve of $\scriptL_2$, and likewise with indices reversed, we have that 
$$
\#\scriptI(\scriptL^0_1,\scriptL_2) \leq \#\scriptL_1^0,
$$
and
$$
\#\scriptI(\scriptL_1,\scriptL_2^0) \leq \#\scriptL_2^0.
$$
Finally, in the nontrivial case $\#\scriptL_1 \geq 1$ (otherwise, there are no incidences), $\#\scriptL_2^0 \lesssim \#\scriptL_1^{\frac1{n-1}}\#\scriptL_2$.

For the remainder of the argument, we may therefore assume that $\scriptL_1$ and $\scriptL_2$ contain only nondegenerate curves (i.e., $\scriptL^0_1 = \scriptL^0_2 = \emptyset$).  In particular, $d\pi_j$  has full rank along all elements of $\scriptL_j$.    

Now let $D$ be a large constant, to be determined. By Theorem 4.1 of \cite{guth2015erdHos}, there exists a polynomial $q$ on $\R^{n-1}$ such that 
$$
\R^{n-1} \setminus Z^q
 = \bigcup_{\alpha \in \scriptA}\scriptO^\alpha,
$$
where $Z^q:=\{q=0\}$, the $\scriptO^\alpha$ are disjoint open sets, $\#\scriptA \lesssim D^{n-1}$, and 
$$
\#\scriptO^\alpha \cap \scriptP_1 \lesssim D^{-(n-1)}\#\scriptP_1, \qquad \alpha \in \scriptA.
$$

With $Q:=q \circ\pi_1$, we decompose
$$
\scriptL_j = \scriptL_j^Q \cup \bigcup_{\alpha \in \scriptA} \scriptL_j^\alpha,
$$
where $\ell \in \scriptL_j^Q$ when $\ell$ intersects $Z^Q = \pi_1^{-1}(Z^q)$, and $\ell \in \scriptL_j^\alpha$ if $\ell$ intersects $\tilde\scriptO^\alpha:=\pi_1^{-1}(\scriptO^\alpha)$.  In particular, curves in $\scriptL_1^Q$ are contained in $Z^Q$ and curves in $\scriptL_1^\alpha$ are contained in $\tilde\scriptO^\alpha$.  We may thus expand
\begin{equation} \label{E:incidence decomp}
\scriptI(\scriptL_1,\scriptL_2)
=
\scriptI(\scriptL_1^Q,\scriptL_2^Q) \cup \bigcup_{\alpha \in\scriptA} \scriptI(\scriptL_1^\alpha,\scriptL_2^\alpha).
\end{equation}

We begin by bounding the ``cell contribution.'' By Proposition~\ref{P:simple est},
\begin{equation} \label{E:one cell}
\#\scriptI(\scriptL_1^\alpha,\scriptL_2^\alpha) \lesssim \#\scriptL_2^\alpha + (\#\scriptL_2^\alpha)^{\frac1{n-1}}\#\scriptL_1^\alpha.
\end{equation}

Let $\ell_2 \in \scriptL_2$.  Then $\deg Q \circ \ell_2 \lesssim D$, so $\ell_2$ can only intersect $Z^Q$ at a bounded number of points, unless $\ell_2 \subset Z^Q$.  Therefore, $\#\{\alpha : \ell_2 \in \scriptL_2^\alpha\} \lesssim D$.  This gives us the estimate
\begin{equation} \label{E:L2alpha}
\sum_{\alpha \in\scriptA} \#\scriptL_2^\alpha \lesssim D\#\scriptL_2.
\end{equation}

We turn to the second term on the right of \eqref{E:one cell}.  Since $d\pi_1$ has full rank on the curves in $\scriptL_1$, by B\'ezout's theorem, each point in $\scriptP_1$ has a bounded number of pre-images in $\scriptL_1$, whence
$$
\#\scriptL_1^\alpha \lesssim \#\scriptP_1 \cap \scriptO^\alpha \lesssim D^{-(n-1)}\#\scriptP_1 \leq  D^{-(n-1)}\#\scriptL_1, \qquad \alpha \in \scriptA.
$$
Now we can use H\"older's inequality to sum the second term in \eqref{E:one cell}:
\begin{equation} \label{E:L1alpha stuff}
\sum_{\alpha \in \scriptA} (\#\scriptL_2^\alpha)^{\frac1{n-1}}\#\scriptL_1^\alpha \lesssim D^{-(n-1)}\#\scriptL_1\#\scriptA^{\frac{n-2}{n-1}}\bigl(\sum_{\alpha \in \scriptA} \#\scriptL_2^\alpha\bigr)^{\frac1{n-1}}.
\end{equation}
Combining \eqref{E:L2alpha} and \eqref{E:L1alpha stuff} gives
$$
\sum_{\alpha \in \scriptA} \scriptI(\scriptL_1^\alpha,\scriptL_2^\alpha) \lesssim D\#\scriptL_2 + D^{-\frac{n-2}{n-1}}\#\scriptL_1\#\scriptL_2^{\frac1{n-1}}.
$$
We optimize by setting $$
D\#\scriptL_2 = D^{-\frac{n-2}{n-1}}\#\scriptL_1\#\scriptL_2^{\frac{1}{n-1}},
$$
which suggests $D=\#\scriptL_1^{\frac{n-1}{2n-3}}\#\scriptL_2^{-\frac{n-2}{2n-3}}$. Of course, we also want $D \geq 1$, and with $D:=\max\{\#\scriptL_1^{\frac{n-1}{2n-3}}\#\scriptL_2^{-\frac{n-2}{2n-3}},1\}$, we have
\begin{equation} \label{E:cells ok}
\sum_{\alpha\in \scriptA}\#\scriptI(\scriptL_1^\alpha,\scriptL_2^\alpha)\lesssim\#\scriptL_2+\#\scriptL_1^\frac{n-1}{2n-3}\#\scriptL_2^\frac{n-1}{2n-3}.
\end{equation}

Now it remains to bound $\#\scriptI(\scriptL_1^Q,\scriptL_2^Q)$. Since $\pi_1(\scriptL_1^Q) \subseteq Z^q$, by Proposition \ref{P:bilinear poly}, 
$$
\#\scriptI(\scriptL_1^Q,\scriptL_2^Q)\lesssim\#\scriptL_1+D\#\scriptL_2\leq \#\scriptL_1 + \#\scriptL_2+\scriptL_1^\frac{n-1}{2n-3}\#\scriptL_2^\frac{n-1}{2n-3},
$$
and thus we conclude the proof of Theorem \ref{T:multijoint}.
\end{proof}

\section{Joints in Manifolds} \label{S:manifolds}

We now wish to consider variants of the problems considered in Theorems~\ref{T:joints} and~\ref{T:multijoint} on manifolds, and hence, necessarily, without the hypothesis that the flow maps are polynomial.  Since flows can intersect infinitely many times in this setting (indeed, consider the helical flow $t \mapsto (\cos t,\sin t,t)$ and the vertical flow $t \mapsto (1,0,t)$ on the cylinder $\mathbb S^1 \times \R$), we need an additional hypothesis.

\begin{theorem}
\label{T:joints in M}
Let $M$ be a smooth, $n$-dimensional manifold, each of whose components is simply connected, and let $\scriptV\leq \scriptX(M)$ be a finite dimensional vector space of smooth vector fields, whose elements generate a nilpotent Lie algebra $\frak g$ of dimension $N$, spanning the tangent space $T_xM$ at each point $x \in M$.  If $\scriptL$ is a finite collection of flows of elements of $\scriptV$, and $\scriptJ$ is the set of joints of $\scriptL$, then 
$$
\#\scriptJ \lesssim \#\scriptL^{n/n-1},
$$  
where the implicit constant depends only on $N$.  
\end{theorem}

\begin{theorem}
\label{T:multijoints in M}
Let $M$ be a smooth, $n$-dimensional manifold, each of whose components is simply connected.  Assume that $X_1,X_2 \in \scriptX(M)$ are smooth vector fields, generating a nilpotent Lie algebra $\mathfrak g$ of dimension $N$, whose elements span $T_xM$ for each $x$. Assume further that for all $x \in\scriptM$, $X_j(x) \not\in [\frak g,\frak g]_{x}$, $j=1,2$.  If $\scriptL_j$ is a finite collection of integral curves of $X_j$, $j=1,2$, and $\scriptI$ denotes the set of their incidences, then
\begin{equation} \label{E:multijoints in M}
\#\scriptI \lesssim \#\scriptL_1 + \#\scriptL_2 + (\#\scriptL_1\#\scriptL_2)^{\frac{n-1}{2n-3}},
\end{equation}
where the implicit constant depends only on $N$.
\end{theorem}

Here we use the notation $[\frakg,\frakg]_x$ to denote the vector space
$$
\{[X,Y](x) : X,Y \in \frakg\}.
$$
We note that $n \leq N$ by the hypothesis that the H\"ormander condition holds.

Our main tool in proving both results is a lifting argument of $\frak g$ to a subset of $\scriptX(\R^n)$ whose elements have polynomial flows.  While the result we need is nearly \cite[Proposition 8.7]{CDSS}, which in turn utilized heavily results from \cite[Chapter 1]{Corwin}, we do not have a result with matching hypotheses that we can directly quote, so we give an essentially self-contained argument for the convenience of the reader.  Before stating our lemma, we note that if, under the hypotheses of Theorem~\ref{T:multijoints in M}, we let $\scriptV$ denote the linear span of $X_1,X_2$, then the hypotheses of Theorem~\ref{T:joints in M} hold, in addition to the hypothesis that $X_1,X_2$ do not lie in the $C^\infty$ module generated by of $\frak g^{(2)}$.  (A pair for which this conditions fails is $\tfrac\partial{\partial x},x\tfrac{\partial}{\partial y} \in \scriptX(\R^2)$.) 

\begin{lemma}\label{L:Lie algebra}
Under the hypotheses of Theorem~\ref{T:joints in M}, there exists a local diffeomorphism $\Phi:M \to \R^n$ such that the pushforwards $\Phi_*X$, $X \in \frak g$ extend to all of $\R^n$ and, moreover, the extensions have polynomial flows.  Under the additional hypotheses of Theorem~\ref{T:multijoints in M}, for $j=1,2$, there exists a choice, depending on $j$, of $\Phi$ such that $\Phi_*X_j = \tfrac{\partial}{\partial x_1}$. 
\end{lemma}

\begin{proof}[Proof of Lemma~\ref{L:Lie algebra}]
As mentioned, we bring together arguments from \cite{CDSS,Corwin}. Without loss of generality, we may assume that $M$ is connected.  

Fix a point $x_0 \in M$, and let $\frak z$ denote the Lie subalgebra of $\frak g$ consisting of those $X \in \frak g$ with $X(x_0) = 0$.  By the H\"ormander condition, $\dim\frak z = N-n$.  

By nilpotency, we may define a chain of subalgebras,
$$
\frak g_M:=\frak z < \frak g_{M-1} < \cdots < \frak g_1:=\frak g,
$$
by setting
$$
\frak g_{j-1}:=\{X \in \frak g : [X,\frak g_j] \in \frak g_j\},
$$
and then reindexing as needed.
We may thus construct a basis
$$
\{Y_1,\ldots,Y_N\}
$$
for $\frak g$ such that 
$$
\{Y_{k_j+1},\ldots,Y_{N}\}
$$
is a basis for $\frak g_j$, where $N-k_j:=\dim\frak g_j$.  By construction, the span of $\{Y_k,\ldots,Y_N\}$ is a Lie subalgebra for each $k$.  (This is called a weak Malcev basis of $\frak g$ through $\frak z$ \cite[Theorem~1.1.13]{Corwin}.)  

The Lie group $G:=\exp \frak g$ is $N$ dimensional and may be given coordinates
$$
(t_1,\ldots,t_N) \mapsto e^{t_1Y_1}e^{t_2Y_2} \cdots e^{t_NY_N}(0).
$$
Multiplication and, thus, exponentiation by elements in $\frak g$ are polynomial of degree at most $N$ in these coordinates by the Baker--Campbell--Hausdorff forumla \cite[Proposition~1.2.8]{Corwin}.  Furthermore, $Z:=\exp \frak z$ has an especially nice expression: 
$$
Z = \{(0',t'') : t'' \in \R^{N-n}\},
$$
and the quotient map $\Pi:G \to G/Z$ is given by 
$$
(t',t'') \mapsto t'.
$$
Quotienting by $Z$ on the right respects left multiplication, so there is a well-defined pushfoward $\Pi_*: \frak g \to \scriptX(\R^n)$ on the left-invariant vector fields, which is also a Lie algebra homomorphism \cite[Lemma~8.2]{CDSS}.  

Let $\hat{\frak g}:=\Pi_*\frak g$, a Lie subalgebra of $\scriptX(\R^n)$.  We denote by $\hat X$ the image of $X$ under $\Pi_*$.  The flows of elements of $\hat{\frak g}$ are, as  compositions of polynomials, polynomial, and, since $\Pi$ is a submersion and $\frak g$ spans $T_xG$ at every point, $\hat {\frak g}$ obeys the H\"ormander condition.  

Now we construct our local diffeomorphism $\Phi:M \to \R^n$.  (It is in this argument that there is some deviation from the argument in \cite{CDSS}, since we do not assume that flows in $M$ of elements of $\scriptV$ are global.)  By our assumption that the H\"ormander condition holds on $M$ and Chow's theorem \cite{chow}, for any $x \in M$, there exist $m$ and $X_1,\ldots,X_m \in \frak g$ such that the iterated flow
$$
e^{X_m}(e^{X_{m-1}}(\cdots(e^{X_1}(x_0))\cdots))
$$
is defined (in $M$) and equals $x$.  Since $M$ is simply connected, the map 
$$
\Phi(x):=e^{\hat X_m}(e^{\hat X_{m-1}}(\cdots(e^{\hat X_1}(0)\cdots))),
$$
is well-defined, even though $x \mapsto (X_1,\ldots,X_m)$ most certainly is not.  By the H\"ormander condition, $\Phi_*$ is surjective, and so $\Phi$ is a local diffeomorphism.  Furthermore, $\Phi_*X = \hat X$, for every $X \in \frak g$.   

Having proved the lemma under the hypotheses of Theorem~\ref{T:joints in M}, we add the remaining hypotheses of Theorem~\ref{T:multijoints in M}.  We fix $j\in\{1,2\}$ and  observe that 
$$
\frak h:=\frak z + \frak g^{(2)}
$$
is a Lie subalgebra (actually, an ideal) of $\frak g$ that does not contain $X_j$.  We may thus construct a weak Malcev basis first of $\frak h$ through $\frak z$, and then extend to a weak Malcev basis of $\frak g$ through $\frak z$ whose first element is $Y_1:=X_j$.  In these coordinates, $X_j = \tfrac{\partial}{\partial x_1}$, as claimed.
\end{proof}

\begin{proof}[Proof of Theorem~\ref{T:joints in M}]
By convexity (i.e., $\frac n{n-1}>1$), we may assume that $M$ is connected.  

Since $\Phi$ might fail to be one-to-one if the flows of vector fields in $\frak g$ are not complete in $M$, we cannot directly apply Theorem~\ref{T:joints}.  However, it is possible to reprove the analogue of  Lemma~\ref{L:Quilodran}, using our lifting $\Phi$.  

To this end, let us suppose that $\emptyset \neq \scriptJ' \subseteq \scriptJ$ has the properties that
\begin{center}
    If $\ell \in \scriptL$ intersects $\scriptJ'$, then $\#\ell \cap \scriptJ' \geq m$; and
    $\#\scriptJ' \ll m^n$.
\end{center} 
By assumption, $\#\hat\scriptJ' \ll m^n$, where $\hat\scriptJ':=\Phi(\scriptJ')$.  Hence, there exists a polynomial $0 \not\equiv Q$, with $\deg Q \ll m$, and $Q \equiv 0$ on $\hat\scriptJ'$.  Let $\scriptL'$ denote the set of lines intersecting $\scriptJ'$ and let $\hat\scriptL':=\Phi_*(\scriptL')$.  By construction, $\ell \in \scriptL'$ cannot be a singleton.  Therefore, $\hat \ell:=\Phi\circ\ell$ is a nonconstant polynomial flow, and thereby one-to-one, so it intersects at least $m$ points of $\hat\scriptJ'$.  Therefore, $Q$ vanishes on $\hat\ell$ at least $m > \deg(Q \circ \hat l)$ times, so $Q \equiv 0$ on $\ell$.  Letting $\hat \scriptV:=\Phi_*(\scriptV)$, at any point of $\hat\scriptJ'$, $\hat XQ = 0$ for a $\hat\scriptV$-spanning set of vector fields, whence $\hat XQ = 0$ for every $\hat X \in \scriptV$.  The rest of the proof of Lemma~\ref{L:Quilodran} proceeds exactly as before.  The second part of the proof of Theorem~\ref{T:joints} (removing bad curves) may be carried out on $M$, since we do not need special coordinates.
\end{proof}

In proving Theorem~\ref{T:multijoints in M}, we will need a weighted version of the Guth--Katz polynomial partitioning argument (proved by a straightforward adaptation of the proof of \cite[Theorem~4.1]{guth2015erdHos}).
This is because, though $\hat\pi_1$ is a polynomial, $\hat\pi_1 \circ \Phi$ is not, so $\hat\pi_1 \circ\Phi$ might not be bounded-to-1 on $\scriptL_1$.  

\begin{lemma}[\cite{guth2015erdHos}] \label{L:weighted Guth-Katz}
Let $\scriptP$ be a finite collection of points in $\R^d$ and let $\mu$ be a finite Borel measure supported on $\scriptP$.  Then for each $D \geq 1$, there exists a polynomial $q \not\equiv 0$ with $\deg q \leq D$ such that 
\begin{equation} \label{E:decomp Rd}
\R^d\setminus Z^q = \bigcup_{\alpha \in \scriptA} \scriptO^\alpha,
\end{equation}
with $Z^q:=\{q=0\}$,  the $\scriptO^\alpha$ disjoint open sets, and $\#\scriptA \lesssim D^{-d}$; and
$$
\mu(\scriptO^\alpha) \lesssim D^{-d}\mu(\scriptP), \qquad \alpha \in \scriptA.
$$
\end{lemma}

\begin{proof}[Sketch proof of Lemma~\ref{L:weighted Guth-Katz}]
Theorem 4.1 of \cite{guth2015erdHos} directly implies the lemma when $\mu$ is counting measure on $\scriptP$.  We briefly note the changes needed to obtain the result for a general finitely supported measure.  The main step from the original proof is \cite[Corollary 4.4]{guth2015erdHos}.  To generalize it, we need to know that, given $M$ finite sets, $\scriptP_1,\ldots,\scriptP_M$, and a finite Borel measure $\mu$ on $\bigcup_j\scriptP_j$, there exists $0 \not\equiv q$, a polynomial of degree $\deg q \lesssim M^{1/d}$, such that, for each $j$,
$$
\mu(\{q > 0\} \cap \scriptP_j) \leq \tfrac12\mu(\scriptP_j), \qquad \mu(\{q < 0\}\cap \scriptP_j) \leq \tfrac12\mu(\scriptP_j).
$$
In its original form, \cite[Corollary~4.4]{guth2015erdHos} was proved by applying the ``polynomial ham sandwich theorem'' of Stone--Tukey \cite{StoneTukey} to shrinking balls centered at the points of the $\scriptP_j$ and taking a limit.  To allow for a general discrete measure, we need only replace the ball centered at $p \in \scriptP_j$ of radius $\eps \searrow 0$ with a ball centered at $p \in \scriptP_j$ of radius $\eps\mu_p^{1/d}$, with $\eps \searrow 0$.  The remainder of the proof is (word-for-word) the same as the original.
\end{proof}

\begin{proof}[Proof of Theorem~\ref{T:multijoints in M}]
We begin by verifying that the simple estimate, inequality \eqref{E:simple est}, holds.  

By symmetry of the hypotheses on $X_1,X_2$, it suffices to verify that \eqref{E:bilinear poly} holds whenever $\hat \pi_1(\hat\scriptL_1)$ is contained in a degree $D$ variety.  Here $\hat \scriptL_1:=\Phi_*(\scriptL_1)$, and we may assume that $\hat X_1 =\tfrac{\partial}{\partial x}$ and  $\hat \pi_1(x_1,x') = x'$. If every $\ell_1 \in \scriptL_1$ contains at least $C$ incidences and every $\ell_2 \in \scriptL_2$ contains at least $CD$ incidences, then (since the pushforwards of the $\ell_j$ are one-to-one, as our hypotheses imply $X_j \neq 0$) every $\hat\ell_1 \in \hat\scriptL_1$ contains at least $C$ incidences and every $\hat \ell_2 \in \hat\scriptL_2$ contains at least $CD$ incidences, at which point we can arrive at a contradiction by exactly the same argument as in the proof of Claim~\ref{L:bad line}.  Once we learn that failure of \eqref{E:bilinear poly}, i.e., \eqref{E:bilin poly false}, implies the existence of a bad curve in $\scriptL_1 \cup \scriptL_2$, we may follow the same stopping time argument as in the proof of Proposition~\ref{P:bilinear poly}, removing bad curves until we arrive at a contradiction.  

We turn now to the adaptation of the main portion of the proof of Theorem~\ref{T:multijoint} to the manifold setting.  We let $\hat\scriptP_1:=\hat\pi_1 \circ \Phi(\scriptL_1)$, a finite collection of points in $\R^{n-1}$, and we define a measure on $\hat\scriptP_1$ by letting $\mu_p$ equal the number of preimages of $p$ in $\scriptL_1$.  Let 
$$
D:=\max\{1,\#\scriptL_1^{\frac{n-1}{2n-3}}\#\scriptL_2^{-\frac{n-2}{2n-3}}\},
$$
and let $q$ be the partitioning polynomial guaranteed by Lemma~\ref{L:weighted Guth-Katz}, with $d=n-1$ and degree bounded by $D$.  

The decomposition \eqref{E:decomp Rd} induces 
$$
\scriptL_j = \scriptL_j^Q \cup \bigcup_{\alpha \in\scriptA}\scriptL_j^\alpha,
$$
in the expected way, with $\scriptL_j^Q$ denoting the set of curves in $\scriptL_j$ along which $Q:=q \circ \hat\pi_1 \circ \Phi$ vanishes at some point and $\scriptL_j^\alpha$ denoting the curves in $\scriptL_j$ that intersect $\Phi^{-1}(\hat\pi_1^{-1}(\scriptO^\alpha))$.  We obtain once more the decomposition of incidences,  \eqref{E:incidence decomp}.

The single-cell estimate \eqref{E:one cell} follows directly from \eqref{E:simple est}, which we have just verified.  An $\ell_2$ that is contained in $\scriptL_2^\alpha$ for $\gg D$ values of $\alpha \in\scriptA$ pushes forward via $\hat\pi_1\circ\Phi$ to a polynomial curve of bounded degree along which $q$ vanishes many times, and which is therefore contained in $Z^q$.  This implies \eqref{E:L2alpha}.  For inequality \eqref{E:L1alpha stuff}, we need the additional ingredient that
$$
\#\scriptL_1^\alpha = \mu(\hat\scriptP_1 \cap \scriptO^\alpha) \lesssim D^{-(n-1)}\mu(\hat\scriptP_1) = D^{-(n-1)}\#\scriptL_1.
$$
With our chosen value of $D$, the cellular estimate \eqref{E:cells ok} follows.  

Finally, the bound for the ``wall contribution'' follows from inequality \eqref{E:bilinear poly}, previously verified.  
\end{proof}

\printbibliography 


\end{document}